\documentclass{amsart}
\newtheorem{Lem}{Lemma}
\newtheorem{Theo}{Theorem}
\newtheorem{Prop}{Proposition}
\newtheorem{Prob}{Problem}
\newtheorem{Cor}{Corollary}

\newcommand{\Z}{\mathbb{Z}}
\newcommand{\N}{\mathbb{N}}

\newcommand{\F}{\mathbb{F}_p}

\newcommand{\normal}{\triangleleft}
\newcommand{\Lie}{\mathfrak{sl}_2}
\begin{document}
\title{The normal growth of linear groups over formal power serieses}
\begin{abstract}
Put $R=\F[[t_1, \ldots, t_d]])$. 
We estimate the number of normal subgroups of  $\mathrm{SL}_2^1(\F[[t_1, \ldots, t_d]])$ for $p>2$, the number of ideals in the Lie algebra $\Lie(R)$, and the number of ideals in the associative algebra $R$.
\end{abstract}
\author[Y. Barnea]{Yiftach Barnea}
\author[J. C. Schlage-Puchta]{Jan-Christoph Schlage-Puchta}

\maketitle

\section{Introduction and results}

For a group $G$ and an integer $n$ let $s_n(G)$ be the number of subgroups of index $n$, and $s_n^\normal(G)$ be the number of normal subgroups of index $n$. For a Lie algebra $L$ over $F_p$ we let $s_n(L)$ be the number of ideals of index $n$ (that is the codimension $c$ satisfies $p^c=n$). For a ring $R$ let $s_n^\normal(R)$ be the number of ideals of index $n$.

For a $d$-generated group $G$ (discrete or profinite), Lubotzky
\cite{Lubotzky} showed that $s_n^\normal(G)\leq n^{2(d+1)(1+\Omega(n))}$, 
where $\Omega(n)$ denotes the number of prime divisors of $n$, counted with 
multiplicity. Mann \cite{Mann} showed that for the free pro-$p$ group with 2 
generators $\widehat{F}$ we have $s_{p^k}^\normal(\widehat{F})\geq p^{ck^2}$ for 
some positive constant $c$. In particular, Lubotzky's bound is sharp up to a 
constant. By now we also know that the lower bound $s_{p^k}^\normal(G)\geq p^{ck^2}$ holds for many groups, including all Golod-Shafarevich groups \cite{large}. 

We know that the normal growth of a group is quite irregular. This is already indicated by Lubotzky's bound, which depends heavily on the number theoretic properties of $n$. It also follows from the fact that $s_n^\normal(\Gamma)$ is polynomially bounded  for all $n$ with the exception of a set of density 0 (see \cite{examples}). Furthermore if there exists a monotonic function $f$, such that for all $\epsilon>0$ and all $x$ we have $|s_n^\normal(\Gamma) - f(n)|<\epsilon f(n)$ for all $n\leq x$ with at most $\epsilon x$ exceptions, then $f$ is asymptotically equal to 0 or 1, see \cite{multiplicative}. When considering normal growth it therefore makes little sense to look at strict asymptotics but at weaker equivalence relations among functions. A common measure is the type of a function. For a function $g:\mathbb{N}\rightarrow\mathbb{R}$ we say that a function $f:\mathbb{N}\rightarrow\mathbb{R}$ is of type $g$, if there exists a constant $C$, such that $f(n)<g(n)^C$ holds for all $n$, and $f(n)^C>g(n)$ holds for infinitely many $n$. Note that type does not define an equivalence relation. For example, the function $f$ which satisfies $f(2^k)=2^k$ for all $k$, and $f(n)=0$ for all $n$ that are not a power of 2, is of type $n$, but the function $g(n)=n$ is not of type $f$.

On the lower end of the spectrum there are many groups which have logarithmic 
normal growth. In fact, for many pro-$p$ groups the number of normal 
subgroup of any given index is bounded. We refer the 
reader to \cite{few} for results on and examples of such groups. We also know many examples of 
polynomial normal growth, including all $p$-adic analytic groups which are not 
of logarithmic growth. In fact, if $G$ is $p$-adic analytic, Grunewald, Segal and Smith \cite{GSS} showed 
that $\sum s_n^\normal(G) p^{-s}$ is a rational function. If this function has a pole with positive real part, the 
normal growth is polynomial. If it has a pole with real part 0, but no pole with positive real part, the subgroup growth is bounded by some power of $\log n$.
If all poles would have negative real part, the number of normal subgroups would be ultimately decreasing, which is absurd. If there are no poles at all, the rational function is actually a polynomial, that is, there are only finitely normal subgroup, and $G$ is finite.

However, much less is known for the range strictly between polynomial and $n^{\log n}$. The only general result we are aware of is the following due to Segal \cite{Segal}.

\begin{Theo}[Segal]
Let $G$ be a metabelian group, and put $R=\Z[G/G']$. View $G'$ as an $R$-module $M$, where elements of $G/G'$ act on $G'$ via conjugation. Put $\overline{R}=R/\mathrm{ann}_R(M)$, and let  $d$ the Krull dimension of $\overline{R}$. Then there 
exist positive constants $c_1, c_2$, such that
\[
p^{c_1k^{2-\frac{2}{d}}}\leq s_{p^nk}(G)\leq p^{c_2k^{2-\frac{1}{d}}}.
\]
\end{Theo}
Here we consider the normal growth of $\mathrm{SL}_2^1(R)$, where $R=\F[x_1, \ldots, x_d]$, and $\mathrm{SL}_2^1$ is the set of all invertible matrices congruent to the unit matrix modulo the maximal ideal $\mathfrak{m}=(x_1, \ldots, x_d)$. Normal subgroups of $\mathrm{SL}_2^1(R)$ are directly linked to ideal of the associated Lie algebra $\Lie(R)$, and to ideals of $R$. Define $s_{p^k}^\normal(\Lie(R))$ as the number of ideals of $\Lie(R)$ of codimension $k$, and $s_{p^k}^\normal(R)$ as the number of ideals of $R$ of index $p^k$. Then we have the following.
\begin{Theo}
Let $R=\F[[x_1, \ldots, x_d]]$. 
\begin{enumerate}
\item For every $d$ there exists a constant $C$ such that
\[
s_{p^k}^\normal(R) \leq s_{p^{3k}}^\normal(\Lie(R)), s_{p^{3k}}^\normal(\mathrm{SL}_2^1(R)) \leq s_{p^k}^\normal(R) p^{C k^{2-\frac{2}{d}}} 
\]
\item We have $\log s_{p^k}^{\triangleleft}(R)\asymp k^{2-\frac{1}{d}}$.
\item If $d=2$, then $\log s^{\triangleleft}_{p^k}(R)=\frac{2^{3/2}}{3^{5/2}} k^{3/2}+\mathcal{O}(k)$.
\end{enumerate}
\end{Theo}


\begin{Cor}
For every $\epsilon>0$ there exists a pro-$p$ group $G$, such that $s_{p^k}(G)$ is of type $p^{k^2}$, and $s_{p^k}^\normal(G)$ is of type $p^{k^{2-\epsilon}}$.
\end{Cor}

Shalev showed that a pro-$p$ group with subgroup growth bounded by $p^{\left(\frac{1}{8}-\epsilon\right)n^2}$ for some positive $\epsilon$ is $p$-adic analytic, and has therefore polynomial subgroup growth and at most polynomial normal growth. 
This leads us to the following problem.

\begin{Prob}
\begin{enumerate}
\item What is the minimal possible subgroup growth of a group resp. a pro-$p$ group with large normal growth? 
\item Is there a Golod-Shafarevic group of subexponential subgroup growth?
\end{enumerate}
\end{Prob}
 %

\section{Comparison of the growth functions}



In this section we show that the growth functions $s_{p^k}^\normal(\mathrm{SL}_2^1(R))$, $s_{p^k}^\normal(\Lie(R))$, and $s_{p^k}^\normal(R)$ are of the same type.

\begin{Prop}
\label{Prop:Comparison}
There exists a constant $C$, depending on $d$, such that 
\[
s_{p^k}^\normal(R) \leq  s_{p^{3k}}^\normal(\mathrm{SL}_2^1(R)), s_{p^{3k}}^\normal(\Lie(R)) \leq p^{C k^{2-\frac{2}{d}}} s_{p^k}^\normal(R).
\]
\end{Prop}

The lower bound is clear, as for every ideal $I$ we can construct the principal congruence subgroup $\mathrm{SL}_2^1(I)$ and the Lie ideal $\Lie(I)$, and the index of $\mathrm{SL}_2^1(I)$ in $\mathrm{SL}_2^1(\mathfrak{m})$ equals $|\mathfrak{m}/I|^3$, and the same relation holds for the Lie algebra. For the upper bound we first show that every normal subgroup and every Lie ideal is close to a congruence subgroup.

\begin{Prop}
\label{Prop:ideal sandwich}
Let $R=\F[[t_1, \ldots, t_n]]$-algebra, $\mathfrak{m}=(t_1, \ldots, t_n)$ the maximal ideal of $R$. Then for every Lie-ideal 
$\mathfrak{J}$ of $\Lie(\mathfrak{m})$ there exists an ideal $J$ of $R$, such that
$\Lie(J)\geq \mathfrak{J}\geq\Lie(\mathfrak{m}^3J)$.
\end{Prop}
\begin{proof}
Let $J$ be the ideal generated by all entries of elements of $\mathfrak{J}$. Then we clearly have $\Lie(J)\geq\mathfrak{J}$. Now assume that $\begin{pmatrix}a&b\\c&d\end{pmatrix}\in \mathfrak{J}$. We have for $f\in\mathfrak{m}$
\[
\left[\left[\begin{pmatrix}a&b\\c&d\end{pmatrix}, \begin{pmatrix} f&0\\0&-f\end{pmatrix}\right], \begin{pmatrix} 0&t_i\\0&0\end{pmatrix}\right] = \begin{pmatrix} -2ft_ic & 0\\ 0&2ft_ic\end{pmatrix},
\]
and taking commutators with $ \begin{pmatrix} 0&t_j\\0&0\end{pmatrix}$, $ \begin{pmatrix} -t_j&0\\0&t_j\end{pmatrix}$ instead we find that all matrices with entries in $(t_it_jt_kc)$ are in $\mathfrak{J}$. The same can be done for the other positions, and we obtain that $\Lie(\mathfrak{m}^3 J)\subseteq\mathfrak{J}$.
\end{proof}
\begin{Prop}
\label{Prop:normal sandwich}
Let $R=\F[[t_1, \ldots, t_n]]$-algebra, $\mathfrak{m}=(t_1, \ldots, t_n)$ the maximal ideal of $R$. Then for every normal subgroup
$N$ of $\mathrm{SL}_2^1(\mathfrak{m})$ there exists an ideal $J$ of $R$, such that
$\mathrm{SL}_2^1(J)\geq N\geq\mathrm{SL}_2^1(\mathfrak{m}^3J)$.
\end{Prop}
\begin{proof}
Pick a normal subgroup $N$ in $\mathrm{SL}_2^1(R)$.
Let $L$ be the function mapping normal subgroups of $\mathrm{SL}_2^1(R)$ to the associated Lie ideal in $\Lie(R)$. 
Then there exists an ideal $J\normal R$ such that $\Lie(J)\geq L(N)\geq\Lie(\mathfrak{m}^3J)$. Now we have
\[
L(\mathrm{SL}_2^1(\mathfrak{m}^3J)) \geq L(\mathrm{SL}_2^1(\mathfrak{m}^3J)\cap N) = L(\mathrm{SL}_2^1(\mathfrak{m}^3J)) \cap L(N) = L(\mathrm{SL}_2^1(\mathfrak{m}^3J)), 
\]
that is, $L(\mathrm{SL}_2^1(\mathfrak{m}^3J)\cap N)  = L(\mathrm{SL}_2^1(\mathfrak{m}^3J))$. As $\mathrm{SL}_2^1(\mathfrak{m}^3J)\cap N \leq \mathrm{SL}_2^1(\mathfrak{m}^3J)$, we obtain
$\mathrm{SL}_2^1(\mathfrak{m}^3J)\cap N = \mathrm{SL}_2^1(\mathfrak{m}^3J)$, thus, $\mathrm{SL}_2^1(\mathfrak{m}^3J)\leq N$.
\end{proof}

Next we bound the number of generators of a normal subgroup, a Lie ideal and an ideal, respectively. We start by considering ideals of $R$.

We will make use of Gr\"obner bases. Fix a term order $<$, that is, an ordering of the monomials $t_1^{e_1}\cdots t_d^{e_d}$, such that 1 is minimal among all polynomials, and for all monomials $\vec{t}^{\vec{e}}, \vec{t}^{\vec{f}}, \vec{t}^{\vec{g}}$ we have $\vec{t}^{\vec{e}} < \vec{t}^{\vec{f}}\Rightarrow \vec{t}^{\vec{e}+\vec{g}} < \vec{t}^{\vec{f}+\vec{g}}$. For a polynomial $\sum_{\vec{e}} a_{\vec{e}} \vec{t}^{\vec{e}}$ we define the leading term as the summand for which $\vec{t}^{\vec{e}}$ is minimal among all summands with $a_{\vec{e}}\neq 0$. For a set $A\subseteq k[t_1, \ldots, t_d]$ define $LT(A)$ to be the set of all leading terms of elements of $A$. If $I\normal k[t_1, \ldots, t_d]$, then a set $B$ is called a Gr\"obner base von $I$, if $B$ generates $I$ as an ideal, and $LT(B)$ generates $LT(I)$ as an ideal. The important property of Gr\"obner bases is their existence.

\begin{Theo}[Buchberger]
Every ideal in $k[t_1, \ldots, t_d]$ has a Gr\"obner base.
\end{Theo}

We cannot directly apply Gr\"obner bases to our setting, as Gr\"obner bases are defined for polynomial rings, whereas we have to deal with power series, however, the difference is negligible.

\begin{Lem}
Let $I\normal k[[t_1, \ldots, t_d]]$ be an ideal, and put $I_0=I\cap k[t_1, \ldots, t_d]$.
\begin{enumerate}
\item $I_0$ is an ideal of $k[t_1, \ldots, t_d]$;
\item If $k[[t_1, \ldots, t_d]]/I$ is finite dimensional, then $k[[t_1, \ldots, t_d]]/I\cong k[t_1, \ldots, t_d]/I_0$;
\end{enumerate}
\end{Lem}
\begin{proof}
If $P\in I\cap k[t_1, \ldots, t_d]$, and $Q\in k[t_1, \ldots, t_d]$, then $PQ\in I$, as $Q\in k[[t_1, \ldots, r_d]]$. and $PQ\in k[t_1, \ldots, t_d]$, hence, $PQ\in I_0$. We conclude that $I_0$ is an ideal in $k[t_1, \ldots, t_d]$.

If $k[[t_1, \ldots, t_d]]/I$ is finite dimensional, so is the restriction to $k[[t_1]]/(I\cap k[[t_1]])$. Hence, there is some $e_1$, such that $t_1^{e_1}\in I$. The same holds for the other variables, hence, we have $(t_1^{e_1}, t_2^{e_2}, \ldots, t_d^{e_d})\subseteq I_0 \subseteq I$. As $k[[t_1, \ldots, t_d]]/(t_1^{e_1}, t_2^{e_2}, \ldots, t_d^{e_d}) \cong k[[t_1, \ldots, t_d]]/(t_1^{e_1}, t_2^{e_2}, \ldots, t_d^{e_d})$, our claim follows.
\end{proof}

We will therefore talk about Gr\"obner bases of ideals in $k[[t_1, \ldots, t_d]]$.
\begin{Lem}
\label{Lem:bound for generators}
Let $I$ be an ideal in $R=\F[[x_1, \ldots, x_d]]$ of index $p^n$. Then as an ideal $I$ can be generated by $\mathcal{O}(n^{\frac{d-1}{d}})$ elements.
\end{Lem}
\begin{proof}
We first proof the result for ideals generated by monomials. In this case an ideal in the ring $R$ is the same as an ideal in the semigroup $\N^d$. We now prove our claim by induction over $d$. For $d=1$ our claim is trivial. Now assume our claim holds for $d-1$, and let $\vec{n}=(n_1, \ldots, n_d)$ be an element of minimal total degree $t$ in $I$. Then for all generators $\vec{m}=(m_1, \ldots, m_d)$ different from $\vec{n}$ there exists some $i$, such that $n_i>m_i$. It therefore suffices to bound the number of generators such that $n_d>m_d$. For each fixed $n'$ we have that $I_{n'}=\{(n_1, \ldots, n_{d-1}): (n_1, \ldots, n_{d-1}, n')\in I\}$ is an ideal in $\N^{d-1}$. By our inductive assumption the number of generators of this ideal is at most $C_{d-1}(\N\setminus I_{n'})^{\frac{d-2}{d-1}}$. We conclude that the number of generators of $I$ is at most $dC_{d-1}\sum_{\nu=0}^{n_d-1} (\N\setminus I_\nu)^{\frac{d-2}{d-1}}$.

The index of $I$ is at least equal to the number of monomials of total degree $<t$, which equals the number of lattice points in the simplex $x_i\geq 0, \sum x_i<t$. Hence, the inde of $I$ is $>\frac{1}{d!}t^d$. On the other hand the index of $I$ is at least equal to the sum of the indices of the $I_\nu$, hence, the index of $I$ is at least $\max\left(\frac{1}{d!}t^d, \sum_{\nu=0}^{n_d-1} (\N\setminus I_\nu)\right)$. We now use Hölder's inequality with to obtain
\begin{multline*}
\sum_{\nu=0}^{n_d-1} (\N\setminus I_\nu)^{\frac{d-2}{d-1}} 
= \sum_{\nu=0}^{n_d-1} 1\cdot (\N\setminus I_\nu)^{\frac{d-2}{d-1}} 
\leq \left(\sum_{\nu=0}^{n_d-1} 1\right)^{\frac{1}{d-1}}\left(\sum_{\nu=0}^{n_d-1} (\N\setminus I_\nu)\right)^{\frac{d-2}{d-1}}\\
\leq t^{\frac{1}{d-1}}\left(\sum_{\nu=0}^{n_d-1} (\N\setminus I_\nu)\right)^{\frac{d-2}{d-1}} \leq (d!|\N\setminus I|)^{\frac{1}{d(d-1)}} |\N\setminus I|^{\frac{d-2}{d-1}} = (d!)^{\frac{1}{d(d-1)}}|\N\setminus I|^{\frac{d-1}{d}}.
\end{multline*}
We conclude that $I$ has hat most $C_d|\N\setminus I|^{\frac{d-1}{d}}$ generators, where $C_d= d (d!)^{\frac{1}{d(d-1)}} C_{d-1}$.
Hence, our claim holds for ideals generated by monomials.

Now let $I$ be an arbitrary ideal, $\mathcal{G}$ be a Gröbner basis of $I$. Then the number of generators of $I$ is at most $|\mathcal{G}|$. At the same time we know that the leading monomials of $\mathcal{G}$ generate an ideal of the same index as $I$, hence, $|\mathcal{G}|$ is bounded by $C_d|\N\setminus I|^{\frac{d-1}{d}}$ in view of the special case already proven.
\end{proof}
\begin{Cor}
Let $I$ be a proper ideal of finite index in $R$, such that as an $\F$-vector space we have $\dim R/I=n$. Then we have $\Lie(I)/\Lie(I\mathfrak{m})\cong\F^m$, where $m=\mathcal{O}\left(n^{\frac{d-1}{d}}\right)$.
\end{Cor}
\begin{Cor}
There exists for a constant $C$, depending on $d$, such that for every 
proper ideal $I$ of finite index in $R$, such that as an $\F$-vector space we have $\dim R/I=n$, we have that 
the number of normal subgroups $N$ satisfying $\mathrm{SL}_2^1(I)\geq N\geq\mathrm{SL}_2^1(\mathfrak{m}^3I)$ and the number of Lie ideals $\mathfrak{J}$ with $\Lie(I)\geq\mathfrak{J}\geq\Lie(\mathfrak{m}^3I)$ is $\mathcal{O}(p^{Cn^{2-\frac{2}{d}}})$.
\end{Cor}
\begin{proof}
Let $I$ be an ideal with $|R/I|=p^n$. If $\mathfrak{J}$ is a Lie ideal with $\Lie(I)\geq\mathfrak{J}\geq\Lie(\mathfrak{m}^3I)$ then the number of generators of $\mathfrak{J}$ as a Lie ideal is bounded by the number of $\Lie(\mathfrak{m}^3I)$ plus the dimension of $\mathfrak{J}/\Lie(\mathfrak{m}^3I)$. Hence, the number of generators is $\mathcal{O}p^{Cn^{1-\frac{1}{d}}})$. Hence, $\mathfrak{J}$ is determined by one of the $\mathcal{O}(p^{Cn^{2-\frac{2}{d}}})$ subspaces of $\Lie(I)/\Lie(\mathfrak{m}^3I)$.

Similarly every normal subgroup $N$ with $\mathrm{SL}_2^1(I)\geq N\geq\mathrm{SL}_2^1(\mathfrak{m}^3I)$ is generated by $\mathcal{O}(p^{Cn^{1-\frac{1}{d}}})$ elements, and $N$ is determined by one of $\mathcal{O}p^{n^{1-\frac{1}{d}}})$ chains of subgroups $\mathrm{SL}_2^1(I)>N_1>N_2>\dots>N$.
\end{proof}

\section{The number of ideals: the general case}

\begin{Lem}
\label{Lem:lower bound}
We have $s_{p^n}^\normal(R)\geq p^{\frac{2^{3/2}}{3^{3/2} (d-1)!}n^{2-\frac{1}{d}}+\mathcal{O}(n^{2-\frac{2}{d}})}$.
\end{Lem}
\begin{proof}
Put $m=\lfloor \left(\frac{2}{3}n\right)^{\frac{1}{d}}\rfloor$, where $\alpha$ will be chosen later.
Let $\overline{I}$ be the ideal generated by all monomials $\prod_{i=1}^d
x_i^{e_i}$, where $\sum_{i=1}^d e_i=m$ and $\sum_{i=2}^d e_i\geq 1$, together 
with $x_1^{\tilde{n}}$, where $\tilde{n}=n-\binom{m+d}{d}+m$. Clearly 
$|R/\overline{I}|=p^n$, and
\[
\tilde{n}\geq\left(1-\frac{(2/3)^{d/2}}{d!}\right)n + \mathcal{O}(m^{d-1}) \geq\frac{2}{3}n+\mathcal{O}(n^{1-\frac{1}{d}}).
\]
We now construct ideals $I$, such that the ideal generated by the leading terms 
of elements of $I$ subject to a term order which is a refinement of the total 
degree equals $\overline{I}$. Let $X$ be the set of all 
tuples $\vec{e}=(e_1, \ldots, e_d)$ with $\sum_{i=1}^d e_i=m$ and $\sum_{i=2}^d 
e_i\geq 1$. Consider a map $\phi:X\rightarrow \F[x_1]$, such that $\phi(\vec{e})$ 
is of degree $<\tilde{n}$ and $\phi(\vec{e})$ has a zero of multiplicity $m+1$ at $x_1=0$. Put
\[
I_\phi=\left.\left(\prod_{i=1}^d x_i^{e_i} + \phi((e_1, \ldots, e_d))\right| \vec{e}\in X\right)
\]
It is clear that the leading coefficients of elements in $I$ generate $\overline{I}$.

Now suppose that $\phi, \psi:X\rightarrow \F[x_1]$ satisfy the degree conditions and $I_\phi=I_\psi$. If $\phi\neq\psi$, pick a tuple $\vec{e}\in X$, with $P=\phi(\vec{e})\neq\psi(\vec{e})=Q$. Then we have
\[
P-Q=\left(\vec{x}^{\vec{e}}+P\right)-\left(\vec{x}^{\vec{e}}+Q\right)\in I_\phi-I_\psi=I_\phi-I_\phi\subseteq I_\phi,
\]
that is, $I_\phi$ contains $x_1^\ell$ for some $\ell<\tilde{n}$, which is impossible. We conclude that the map $\phi\mapsto I_\phi$ is injective. The number of maps $\phi:X\rightarrow\F[x_1]$ satisfying the degree conditions is 
\[
p^{(\tilde{n}-m-1)X} \geq p^{\left(\frac{2}{3}n+\mathcal{O}(n^{1-\frac{1}{d}}\right)\frac{m^{d-1}+\mathcal{O}(m^{d-2})}{d!}} = p^{\frac{2}{3\cdot (d-1)!} n^{2-\frac{1}{d}}+\mathcal{O}(n^{2-\frac{2}{d}})}.
\]
This implies our claim.
\end{proof}
We have chosen the parameters in such a way that the result is optimal for $d=2$. For higher values of $d$ we do not have an optimal upper bound, so optimizing the lower bound seems without merit.

We now turn to the upper bound.

Let $k$ be a finite field, $R=k[[x_1, \ldots, x_d]]$ as a topological ring. Let $R_n$ be the ring generated by all monomials $\prod_{i=1}^d x_i^{e_i}$, $\sum e_i\geq n$. Let $I$ be a closed ideal, such that $R/I$ is finite. Put $V_n=(R_n\cap I)R_{n+1}/R_{n+1}$, $d_n=\dim_k V_n$. Multiplication with a variable $x_i$ defines an injection $\varphi_i:V_n\rightarrow V_{n+1}$.  Put $W_{n+1}=\langle \varphi_1(V_n), \ldots, \varphi_d(V_n)\rangle$, $e_n=\dim(V_n/W_n)$. 

\begin{Lem}
\label{Lem:general number}
The number of ideals realising a given set of parameters $(e_n), (d_n)$ is at most
\[
\prod_{n=1}^\infty \begin{pmatrix}\binom{n+d}{d}-d_n+e_n\\ e_n\end{pmatrix}_p\exp_p\left(\sum_{n=1}^\infty e_n\sum_{m=n+1}^\infty \binom{m+d}{d}-d_m\right).
\]
\end{Lem}
\begin{proof}
For each $n$ pick a subspace of $(R_n/R_{n-1})/W_n$ of dimension $e_n$.  As $\dim R_{n+1}/R_n=\binom{n+d}{d}$, this can be done in $\begin{pmatrix}\binom{n+d}{d}-d_n+e_n\\ e_n\end{pmatrix}_p$ ways. Next pick a basis $b_1, \ldots, b_{e_n}$ of this space. Using Gaussian elimination we can guarantee that all coefficients in $V_m$, $m>n$ of $b_i$ vanish. So in each $R_m/R_{m-1}$, $m>n$ we can pick $\dim R_m/R_{m+1} - \dim V_m$ coefficients of $b_i$. 
\end{proof}

If $R/I=p^N$, then $\sum_{n=1}^\infty \binom{n+d}{d}-d_n=N$, and if 
$d_n=\binom{n+d}{n}$, then $V_n=R_n/R_{n+1}$, and therefore 
$W_{n+1}=R_{n+1}/R_{n+2}$, that is, $d_m=\binom{m+d}{d}$ for all $m\geq n$. 
Hence the number of possible choices for the sequence $d_n$ is bounded by the 
number of ordered partitions of $N$, which is $2^{N-1}$. A factor of this size 
is clearly negligible. Furthermore $\sum_{n=1}^\infty e_n$ equals the number 
of generators of $I$, which by Lemma~\ref{Lem:bound for generators} is 
$\mathcal{O}(N^{\frac{d-1}{d}})$. Hence, given a sequence $d_n$, the number of 
possible choices for the sequence $e_n$ is bounded by the number of ordered 
partitions of $CN^{\frac{d-1}{d}}$ into $N$ non-negative parts, which is 
\[
\binom{N+CN^{\frac{d-1}{d}}}{N}\leq 
(N+CN^{\frac{d-1}{d}})^{CN^{\frac{d-1}{d}}} \leq\exp\left(C' N^{\frac{d-1}{d}}\log N\right),
\]
which is even smaller. We conclude that it suffices to consider the maximum over  all pairs of sequences that actually occur as the parameters of an ideal.

We have $\binom{n}{k}_p\ll p^{k(n-k)}$, hence,
\[
\prod_{n=1}^\infty \begin{pmatrix}\binom{n+d}{d}-d_n+e_n\\ e_n\end{pmatrix}_p \leq C^{\#\{n:\binom{n+d}{d}\neq d_n\}} \exp_p\left(\sum_{n=1}^\infty e_n\left(\binom{n+d}{d}-d_n\right)\right).
\]
The index of the ideal described by the sequences $(e_n)$, $(d_n)$ is 
$p^{\sum \binom{n+d}{d}-d_n}$, in particular
$C^{\#\{n:\binom{n+d}{d}\neq d_n\}}<p^{C'N}$, which is negligible. We see that it suffices to prove
\[
\sum_{n=1}^\infty e_n\sum_{m=n}^\infty \binom{m+d}{d}-d_m \ll \left(\sum_{n=1}^\infty \binom{n+d}{d}-d_n\right)^{2-\frac{1}{d}}.
\]
However, this follows immediately from the fact that $\sum_1^\infty e_n\ll N^{1-\frac{1}{d}}$ and $\sum_{m=1}^\infty \binom{m+d}{d}-d_m=N$.

\section{The case $d=2$}

From now on we consider only the case $d=2$.

\begin{Lem}
\label{Lem:generator growth}
Suppose that $V_n\neq \{0\}$. Then we have $\dim W_{n+1}>\dim V_n$.
\end{Lem}
\begin{proof}
As $\varphi_1, \varphi_2:V_n\rightarrow R_{n+1}/R_{n+2}$ are injective, it suffices to show that $\varphi_1(V_n)\neq\varphi_2(V_n)$. Let $x_1^ax_2^{n-a}$ be the monomial with the smallest exponent of $x$ occurring in $V_n$. Then in all monomials in $\varphi_1(V_n)$ the exponent of $x_1$ is at least $a+1$. However, $\varphi_2(V_n)$ contains an element which contains the monomial $x_1^a x_2^{n+1-a}$, therefore $\varphi_1(V_n)\neq\varphi_2(V_n)$.
\end{proof}

\begin{Lem}
\label{Lem:d=2 formula}
The number of ideals of index $|k|^N$ in $k[[x_1, x_2]]$ is
\[
\exp_p\left(\underset{d_n>0\Rightarrow d_{n+1}\geq d_n+1}{\underset{d_{n+1}\leq n+1} {\underset{\sum (n+1-d_n)=N} {\max_{(d_n) }}}}
\sum_{n:d_n>0} (d_n-d_{n-1}-1)\sum_{m>n}(m+1-d_m) + \mathcal{O}(N)\right)
\]
\end{Lem}
\begin{proof}
To determine a sequence $(d_n)$ satisfying all properties we first pick an integer $n_0$, which is the smallest $n$ 
such that $d_n>0$. Then $\sum_{n<n_0}(n+1-d_n)=\sum_{n<n_0}(n+1)=\binom{n_0+1}{2}$. For $n\geq n_0$ we have 
that $n+1-d_n$ is non-increasing, and any non-increasing sequence with sum $N-\binom{n_0+1}{2}$ defines a 
hence, the sequence $d_n$ is uniquely determined by $n_0$ and a ordered partition of $N-\binom{n_0+1}{2}$. 
Therefore the number of choices for the sequence $(d_n)$ is
\[
\sum_{n_0\geq 0} p\left(N-\binom{n_0+1}{2}\right)<\sqrt{2N}p(N)=\mathcal{O}(e^{c\sqrt{N}}).
\]
We conclude that the sum over all possible sequences can be replaced by the maximum.

By Lemma~\ref{Lem:generator growth} we can bound $e_n$ by $d_n-d_{n-1}-1$ if $d_{n-1}\neq 0$. If $d_{n-1}=0$, then replacing $e_n$ by $d_n-1$ introduces an error of 1, but this error is multiplied by something less than the index, and is captured by the error term.

Finally the subspace numbers are
\[
\prod_{n=1}^\infty \binom{n+1-d_n+e_n}{e_n}_p \leq \prod_{n=1}^\infty \binom{n+1-d_{n-1}}{e_n}_p.
\]
The sequence $n-d_n$ is non-increasing, hence, the number of possibly choices for the subspaces is at most $\prod_{n=1}^\infty \binom{n_0+1}{e_n}_p$, where $n_0$ is the smallest integer with $d_{n_0}>0$.

For $n>n_0$ we have $e_n\leq d_n-d_{n-1}-1$, hence, if we denote by $n-1$ an index such that $d_{n_1}=n_1+1$, we have
\[
\sum e_n \leq e_{n_0} + \sum_{n>n_0} d_n-d_{n-1}-1 = d_{n_0} + \sum_{n=n_0+1}^{n_1} d_n-d_{n-1}-1 = d_{n_1}-(n_1-n_0) = n_0+1.
\]
We conclude that 
\[
\prod_{n=1}^\infty \binom{n+1-d_{n-1}}{e_n}_p \leq \prod_{n=1}^\infty \binom{n_0+1}{e_n}_p \leq \exp_p\left(\sum_{n=1}^\infty e_n(n_0+1)\right)\leq p^{n_0^2+n_0}.
\]
There are $\binom{n_0}{2}$ monomials with sum of exponent $<n_0$, hence, $n_0\leq \sqrt{2N}+1$, and we see that the number of subspaces is $\exp_p(\mathcal{O}(N))$. We find that the stated expression is an upper bound for the number of ideals.

To show that we actually have an equality we have to show that for every sequence $(d_n)$ with $d_n\leq n+1$ and $d_{n+1}\geq d_n+1$ for all $n$ with $d_n>0$ there exists an ideal which realizes these dimensions and satisfies $e_n=d_n-d_{n-1}-1$ for all $n$. Such an ideal is generated by the monomials $X=\{x^a y^b : a\leq d_{a+b}-1\}$. Here $(I\cap R_n)R_{n+1}/R_{n+1} = \langle x^ay^{n-a} : a\leq d_n-1\rangle$, which is of dimension $n$. Moreover,
\begin{eqnarray*}
\varphi_1(X)\cup\varphi_2(X) & = & \{x^{a+1} y^{n-a} : a\leq d_n-1\}\cup\{x^a y^{n+1-a} : a\leq d_n-1\}\\
 & = & \{x^a y^{n+1-a} : a\leq d_n\},
\end{eqnarray*}
which generates a subspace of dimension $d_n+1$. 
\end{proof}

Next we determine the maximum of this expression over all sequences. Put $r_n=n+1-d_n$. Then the index of the ideal becomes $p^{\sum r_n}$, the restrictions become $r_n=n+1$ for $n<n_0$, and $r_{n+1}\leq r_n$ for $n\geq n_0$. The function to be maximized becomes
\[
\sum_{n:d_n>0} (d_n-d_{n-1}-1)\sum_{m>n}(m+1-d_m) = \sum_{n\geq n_0} (r_{n-1}-r_n)\sum_{m>n} r_m.
\]
Put $R_n=\sum_{m>n} r_n$. We now apply partial summation to the expression on the right and obtain
\[
\sum_{n\geq n_0} (n_0-r_n)r_{n+1} = n_0\sum_{n>n_0} r_n - \sum_{n\geq n_0} r_nr_{n+1}.
\]
For fixed $\sum r_n$ this expression becomes maximal when $r_n=1$ for all $n\geq n_0$. In this case we obtain
\[
(n_0-1)\sum_{n>n_0} r_n =(n_0-1)\left(N-\binom{n_0}{2}\right) = N(n_0-1) -\frac{n_0^2(n_0-1)}{2}+\frac{n_0(n_0-1)}{2}.
\]
The derivative of this expression with respect to $n_0$ is
\[
N-\frac{3}{2}n_0^2 +2n_0-\frac{1}{2}.
\]
Hence, the maximum is attained for $n_0=\sqrt{\frac{2}{3}N}+\mathcal{O}(1)$, and the maximal value is
\[
\left(\sqrt{\frac{2}{3}N}+\mathcal{O}(1)\right)\left(N-\frac{1}{3}N+\mathcal{O}(\sqrt{N}\right)=\left(\frac{2}{3}\right)^{3/2}N^{3/2}+\mathcal{O}(N)
\]


\begin{thebibliography}{9}
\bibitem{few} Y. Barnea, N. Gavioli, A. Jaikin-Zapirain, V. Monti, C. M. Scoppola, Pro-$p$ groups with few normal subgroups, {\em J. Algebra} {\bf 321} (2009), 429--449.
\bibitem{large} Y. Barnea, J.-C. Schlage-Puchta, Large normal subgroup growth and large characteristic subgroup growth, {\em J. Group Theory} {\bf 23} (2020),  1--15.
\bibitem{GSS} F. Grunewald, D. Segal, G. Smith, Subgroups of finite index in nilpotent groups, {\em invent. math.} {\bf 93} (1988), 185--223.
\bibitem{Lubotzky} A. Lubotzky, Enumerating boundedly generated finite groups, {\em J. Algebra} {\bf 238} (2001), 194--199.
\bibitem{LubSha} A. Lubotzky, A. Shalev, On some $\Lambda$--analytic pro-$p$ groups, {\em Isr. J. Math.} {\bf 85} (1994), 307-337.
\bibitem{Mann} A. Mann, Enumerating finite groups and their defining relations, {\em J. Group Theory} {\bf 1} (1998), 59--64.
\bibitem{examples} T. W. Müller, J.-C. Schlage-Puchta,
Some examples in the theory of subgroup growth, {\em Monatsh. Math.} {\bf 146} (2005),  49--76.
\bibitem{multiplicative} J.-C. Schlage-Puchta, Weakly multiplicative arithmetic functions and the normal growth of groups., {\em
Arch. Math.} {\bf 112} (2019) 233-240.
\bibitem{Segal} D. Segal, On the growth of ideals and submodules,
{\em J. London Math. Soc.} (2) {\bf 56} (1997), 245--263.

\end{thebibliography}
\end{document}